\documentclass[10pt,notitlepage]{article}
\usepackage{amsmath}
\usepackage{amssymb}
\usepackage{amsthm}
\usepackage{amscd}
\usepackage{latexsym}
\usepackage{graphicx}

\newtheorem{theorem}{Theorem}[section]
\newtheorem{proposition}[theorem]{Proposition}
\newtheorem{lemma}[theorem]{Lemma}
\newtheorem{corollary}[theorem]{Corollary}
\newtheorem{definition}[theorem]{Definition}
\newtheorem{remark}[theorem]{Remark}
\def\Im{\operatorname{Im}}
\def\Ker{\operatorname{Ker}}
\def\Ann{\operatorname{Ann}}
\def\Sign{\operatorname{Sign}}
\def\sign{\operatorname{sign}}
\def\Coker{\operatorname{Coker}}
\def\Int{\operatorname{Int}}
\def\Diff{\operatorname{Diff}}
\begin{document}
\title{A class function on the mapping class group of an orientable surface and the Meyer cocycle}
\author{Masatoshi Sato}
\date{}
\maketitle
\begin{abstract}
In this paper we define a $\mathbf{QP}^1$-valued class function on the mapping class group $\mathcal{M}_{g,2}$ of a surface $\Sigma_{g,2}$ of genus $g$ with two boundary components. Let $E$ be a $\Sigma_{g,2}$ bundle over a pair of pants $P$. Gluing to $E$ the product of an annulus and $P$ along the boundaries of each fiber, we obtain a closed surface bundle over $P$. We have another closed surface bundle by gluing to $E$ the product of $P$ and two disks.

The sign of our class function cobounds the 2-cocycle on $\mathcal{M}_{g,2}$ defined by the difference of the signature of these two surface bundles over $P$.
\end{abstract}
\tableofcontents
\newpage
\section{Introduction}
Let $\Sigma_{g,r}$ be a compact oriented surface of genus $g$ with $r$ boundary components. The mapping class group $\mathcal{M}_{g,r}$ is $\pi_0\Diff_+(\Sigma_{g,r},\partial\Sigma_{g,r})$ where $\Diff_+(\Sigma_{g,r},\partial\Sigma_{g,r})$ is the group of orientation preserving diffeomorphisms of $\Sigma_{g,r}$  which restrict to the identity on the boundary $\partial\Sigma_{g,r}$. We simply denote $\Sigma_g:=\Sigma_{g,0}$ and $\mathcal{M}_g:=\mathcal{M}_{g,0}$. Harer\cite{harer1985shm} proved that
\[
H^2(\mathcal{M}_{g,r};\mathbf{Z})\cong \mathbf{Z} \hspace{1em}g\ge 3,\ r\ge 0, 
\]
see also Korkmaz, Stipsicz\cite{korkmaz2003shg}.
\begin{figure}[h]
  \begin{center}
    \includegraphics{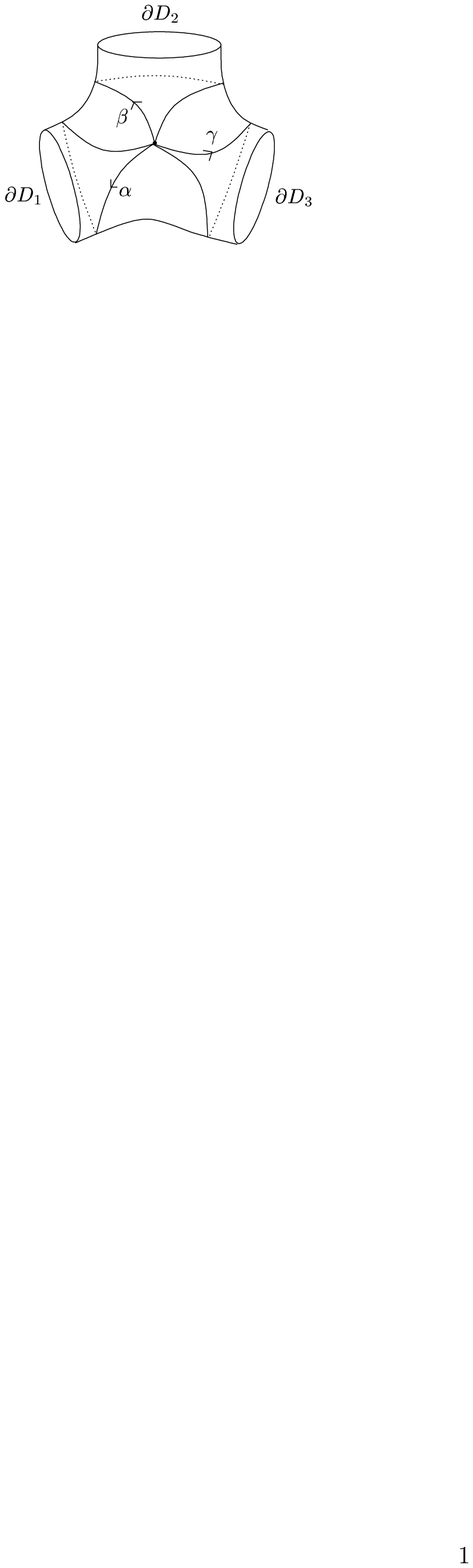}
  \end{center}
  \caption{}
  \label{p.eps}
\end{figure}
Meyer\cite{meyer1973sf} defined a cocycle $\tau_g\in Z^2(\mathcal{M}_g; \mathbf{Z})$ ($g\ge0$) called the Meyer cocycle which represents four times generator of the second cohomology class when $g\ge3$. Let $P:=S^2-\amalg_{i=1}^3 \Int D_i$ where $D_i\subset S^2$ is a disk, $\Int D_i$ its interior in $S^2$, and $\alpha, \beta, \gamma \in \pi_1(P)$ be the homotopy classes as shown in Figure \ref{p.eps}. We consider a  $\Sigma_{g,r}$ bundle $E_{g,r}^{\varphi, \psi}$ on the pair of pants $P$ which has monodromies $\varphi$, $\psi$, $(\psi\varphi)^{-1}\in \mathcal{M}_{g,r}$ along $\alpha$, $\beta$, $\gamma\in \pi_1(P)$. The diffeomorphism type of $E_{g,r}^{\varphi, \psi}$ does not depend on the choice of representatives in the mapping classes $\varphi$ and $\psi$. The Meyer cocycle is defined by
\[
\begin{array}{cccccc}
\tau_g:&\mathcal{M}_{g}&\times&\mathcal{M}_{g}&\to&\mathbf{Z}\ ,\\
&(\ \varphi&,&\psi\ )&\mapsto&\Sign E_{g}^{\varphi,\psi}
\end{array}
\]
where $\Sign E_{g}^{\varphi,\psi}$ is the signature of the compact 4-manifold $E_{g}^{\varphi,\psi}$. For $k>0$, it is known as Novikov additivity that when two compact oriented $4k$-manifolds are glued by an orientation reversing diffeomorphism of their boundaries, the signature of their union is the sum of their signature.  When a pants decomposition of a closed 2-manifold is given, the signature of a $\Sigma_g$ bundle on the 2-manifold is the sum of the signature of the $\sigma_g$ bundles restricted to each pair of pants. Therefore, it is important to study the Meyer cocycle to calculate the signature of compact 4-manifolds. For $g=1, 2$ the Meyer cocycle $\tau_g$ is a coboundary, and the cobounding function of this cocycle is calculated by several authors, for instance, Meyer\cite{meyer1973sf}, Atiyah\cite{atiyah1987ldeta}, Kasagawa\cite{kasagawa2000fmc}, Iida\cite{iida2004ale}. The Meyer cocycle is not a coboundary if genus $g\ge3$, but the cocycle can be a coboundary when it is restricted to a certain subgroup, and calculated by Endo\cite{endo2000mss}, Morifuji\cite{morifuji2003msf}. 

Let $I$ be the unit interval $[0,1]\subset \mathbf{R}$. By sewing a pair of disks onto the surface $\Sigma_{g,2}$ along the boundary, we have $\Sigma_g$. For $h\in\Diff_+(\Sigma_{g,2},\partial\Sigma_{g,2})$, if we extend $h$ by the identity on the pair of disks, we have a self-diffeomorphism of $\Sigma_g$. we denote it $h\cup id_{\amalg_{i=1}^2 D^2}$. By sewing an annulus $S^1\times I$ onto the surface $\Sigma_{g,2}$ along the boundary, we have $\Sigma_{g+1}$. In the same way, if we extend $h\in\Diff_+(\Sigma_{g,2},\partial\Sigma_{g,2})$ by the identity on the annulus, we have a self-diffeomorphism $h\cup id_{S^1\times I}$.

Define the induced homomorphism on the mapping class group by
\[
\begin{array}{cccc}
\theta:&\mathcal{M}_{g,2}&\to&\mathcal{M}_g\\
&[h]&\mapsto&[h\cup id_{\amalg_{i=1}^2 D^2}]
\end{array}
\]
and 
\[
\begin{array}{cccc}
\eta:&\mathcal{M}_{g,2}&\to&\mathcal{M}_{g+1,0}.\\
&[h]&\mapsto&[h\cup id_{S^1\times I}]
\end{array}
\]
Harer\cite{harer1983shg}\cite{harer1985shm} shows that $\theta$ and $\eta$ induce an isomorphism on the second homology classes when genus $g\ge5$, so that $\tilde{\tau}_g=\eta^*\tau_{g+1}-\theta^*\tau_g$ is a coboundary. Powell\cite{powell1978ttm} proved that the first cohomology group $H_1(\mathcal{M}_{g,r}; \mathbf{Z})$ is trivial for $g\ge 3$ and $r\ge 0$, so by the universal coefficient theorem, it follows that the cobounding function of $\tilde{\tau}_g$ is unique.

In this paper we define a $\mathbf{QP}^1$-valued class function $m$ on the mapping class group $\mathcal{M}_{g,2}$ in an explicit way by using information of the first homology group of a mapping torus of $[h]\in \mathcal{M}_{g,2}$, and prove that the sign of the function $m$ cobounds the cocycle $\tilde{\tau}_g=\eta^*\tau_{g+1}-\theta^*\tau_g$. Especially it turns out that the cocycle $\tilde{\tau}_g$ is coboundary for any $g\ge0$.

In section 1, we construct a class function $m$, prove some properties of this function, and calculate the image of the function. In section 2, we prove that the sign of this function cobounds the difference $\tilde{\tau}_g=\eta^*\tau_{g+1}-\theta^*\tau_g$. By the definition of the Meyer cocycle $\tau_g$, $\tilde{\tau}_g(\varphi, \psi)$ is just the difference $\Sign E_{g+1}^{\eta(\varphi),\eta(\psi)}-\Sign E_g^{\theta(\varphi),\theta(\psi)}$, so that we calculate the difference by using the sign of the function $m$. Moreover we compute the other differences of signature $\Sign(E_{g,2}^{\varphi, \psi})-\Sign(E_g^{\theta(\varphi),\theta(\psi)})$ and $\Sign(E_{g+1}^{\eta(\varphi),\eta(\psi)})-\Sign(E_{g,2}^{\varphi, \psi})$ by the function $m$. 
\section{Class function $m:\mathcal{M}_{g,2}\to \mathbf{QP}^1$}
In this section we define the class function on the mapping class group $\mathcal{M}_{g,2}$ stated in Introduction and describe some properties of the function including the nontriviality. 

For $[p:q]$, $[r:s]\in\mathbf{QP}^1$, we define an addition in $\mathbf{QP}^1$ by
\[
[p:q]+[r:s]=
\begin{cases}
[pr:ps+qr], &\textrm{if} \quad [p:q]\ne[0:1] \textrm{ or } [r:s]\ne[0:1]\\
[0:1], &\textrm{if}\quad [p:q]=[r:s]=[0:1].
\end{cases}
\]
The projective line $\mathbf{QP}^1$ forms an additive monoid under this operation with $[1:0]$ the zero element. 

In this section, all (co)homology groups is with $\mathbf{Q}$ coefficients. 
\subsection{Construction of the class function}\label{subsection2.1}
Before constructing the function, we prepare a fact about homology groups of compact 3-manifolds. Let $Y$ be a compact oriented 3-manifold with boundary $\partial Y$ and $i:\partial Y\hookrightarrow Y$ the inclusion map. Consider the commutative diagram
\[
\begin{CD}
H^1(Y) @>i^*>>H^1(\partial Y) @>\delta^*>>H^2(Y,\partial Y)\\
@VV\cap [Y]V @VV\cap [\partial Y]V @VV\cap [Y]V\\
H_2(Y,\partial Y) @>\partial_*>>H_1(\partial Y) @>i_*>>H_1(Y),
\end{CD}
\]
where the upper and lower rows are the exact sequences of a pair $(Y, \partial Y)$, and the vertical maps are the cap products with the fundamental classes of $Y$ and $\partial Y$.
By the diagram and Poincar\'{e} Duality, it follows that the image of $i^*$ is just its own annihilator with respect to the cup product of $H^1(\partial Y)$
\[
\Im i^*=\Ann (\Im i^*).
\]
In particular, we have
\[
\dim\Ker i_*=\dim\Im i^*=\frac{1}{2}\dim H_1(\partial Y).
\]
We define the mapping torus of $\varphi=[h]\in\mathcal{M}_{g,r}$ by 
\[
X^\varphi:= \Sigma_{g,r}\times I/\sim, \quad (x,1)\sim (h(x), 0), 
\]
and
$\pi: X^\varphi\to I/\partial I=S^1$ by the projection $\pi([x,t])=[t]$, where $[x, t]\in X^\varphi$ is the equivalent class of $(x, t)\in\Sigma_{g,r}\times I$, and $[t]\in I/\partial I=S^1$ the equivalent class of $t\in I$. 

The diffeomorphism type of the mapping torus $X^\varphi$ does not depend on the choice of the representative $h$. We fix the orientation on $X^\varphi$ given by the product orientation on $\Sigma_{g,r}\times I$. Let $i_\varphi:\partial X^\varphi\hookrightarrow X^\varphi$ be the inclusion map. In this subsection we denote $\Sigma:=\Sigma_{g,2}$, and if we fix $\varphi\in\mathcal{M}_{g,2}$, then we write simply $X:=X^\varphi$ and $i:=i_\varphi$. Let $S_1$ and $S_2$ be the two boundary components of $\Sigma$, and $[S_k]$\ $(k=1,2)$ the image under the inclusion homomorphism $H_1(S_k)\to H_1(\Sigma)$ of the fundamental homology class. 

We consider $\Sigma$ as a subspace of $X$ by the embedding $\iota :\Sigma\hookrightarrow X$  $x\mapsto [x, 0]$. We choose points $p_1\in S_1$, $p_2\in S_2$, and $p\in S^1$, and orientation-preserving homeomorphisms $\iota_1: S^1\to S_1$ and $\iota_2: S^1\to S_2$. We define singular cochains $f_k: I\to (S_1\amalg S_2)\times S^1=\partial X$  $(k=1,2,3,4)$ by 
\[
f_1(t)=(\iota_1(t), p),\quad f_2(t)=(\iota_2(t), p),\quad f_3(t)=(p_1, t),\ \text{and} \quad f_4(t)=(p_2, t),\ \text{respectively}.
\]
Let $e_k\in H_1(\partial X)$ be the homology class of $f_k$  $(k=1,2,3,4)$. Then the set $\{e_1, e_2, e_3, e_4\}$ forms a basis for $H_1(\partial X)$. 

Now we describe the kernel of the homomorphism $i_*: H_1(\partial X)\to H_1(X)$.
Since $e_1$ and $e_2$ lie in the kernel of $(\pi|_{\partial X})_*$ and $\pi_*(e_3)=\pi_*(e_4)=[S^1]\in H_1(S^1)$, we have
\[\Ker\,i_* \subset \Ker\,(\pi_* i_*)=\mathbf{Q}e_1\oplus\mathbf{Q}e_2\oplus\mathbf{Q}(e_3-e_4). \]
By the definition of the map $f_k$, $(i\circ f_k)_*[S^1]=\iota_*[S_k]$, and so $i_*(e_1+e_2)=\iota_*([S_1]+[S_2])\in H_1(X)$. Since $S_1\cup S_2$ is the boundary of $\Sigma$, we have $[S_1]+[S_2]=0\in H_1(\Sigma)$. Hence
\[\mathbf{Q}(e_1+e_2)\subset \Ker\,i_*.\]
As we saw at the beginning of this subsection, $\dim\Ker\,i_*=\frac{1}{2}\dim H_1(\partial X)=2$.  It follows that $\Ker i_*=\mathbf{Q}(e_1+e_2)\oplus\mathbf{Q}(p(e_3-e_4)+qe_1)$ for some $p$, $q\in \mathbf{Q}$. Now we can define a class function. 
\begin{definition}
For $\varphi\in\mathcal{M}_{g,2}$, we take $p, q\in\mathbf{Q}$ such that $\Ker\, i_{\varphi*} =\mathbf{Q}(e_1+e_2)\oplus\mathbf{Q}(p(e_3-e_4)+qe_1)$. 

We define $m:\mathcal{M}_{g,2}\to\mathbf{QP^1}$ by $m(\varphi)= [p:q]$. 
\end{definition}
\begin{lemma}
For $\varphi, \psi\in\mathcal{M}_{g,2}$, 
\begin{gather*}
m(\psi\varphi\psi^{-1})=m(\varphi).
\end{gather*}
\end{lemma}
\begin{proof}
Define $\Psi: X^\varphi\to X^{\psi\varphi\psi^{-1}}$ by $\Psi(x,t)=(\psi(x),t)$.
Then the following diagram  commutes 
\[
\begin{CD}
H_1(\partial X^\varphi) @>i_{\varphi*}>>H_1(X^\varphi)  \\
@VV\Psi_* V @VV\Psi_* V \\
H_1(\partial X^{\psi\varphi\psi^{-1}}) @>i_{\psi\varphi\psi^{-1}}*>> H_1(X^{\psi\varphi\psi^{-1}}).
\end{CD}
\]
We can see from the diagram, $\Psi_*$ gives the natural isomorphism between $\Ker(H_1(\partial X^\varphi)\to H_1(X^\varphi))$ and $\Ker(H_1(\partial X^{\psi\varphi\psi^{-1}})\to H_1(X^{\psi\varphi\psi^{-1}}))$. Hence we have $m(\psi\varphi\psi^{-1})=m(\varphi)$. 
\end{proof}
\subsection{Some properties and the nontriviality of the class function} 
By the Serre spectral sequence, we have the exact sequence
\[
\begin{CD}
0@>>>\Coker(\varphi_*-1)@>\iota_*>>H_1(X)@>\pi_*>>H_1(S^1)@>>>0,
\end{CD}
\]
where $\Coker(\varphi_*-1)$ is the cokernel of the homomorphism $\varphi_*-1:H_1(\Sigma)\to H_1(\Sigma)$.

Then we have a unique homomorphism $j_\varphi: \mathbf{Q}e_1\oplus\mathbf{Q}e_2\oplus\mathbf{Q}(e_3-e_4) \to \Coker(\varphi_*-1)$ such that the diagram with exact rows 
\[
\begin{CD}
0 @>>> \mathbf{Q}e_1\oplus\mathbf{Q}e_2\oplus\mathbf{Q} (e_3-e_4) @>>> H_1(\partial X) @>\pi_*>> H_1(S^1) @>>> 0\\
@.@VVj_\varphi V @VVi_*V @| @.\\
0 @>>> \Coker(\varphi_*-1) @>\iota_*>>H_1(X) @>\pi_*>>H_1(S^1)@>>>0
\end{CD}
\]
commutes. By the diagram, we have
\[\Ker i_*=\Ker j_\varphi\, ,\ \text{and}\]
\[
j_\varphi(e_1)=-j_\varphi(e_2)=[S_1]\in \Coker(\varphi_*-1).
\]
Now we introduce a cochain $\omega_l\in C^1(\mathcal{M}_{g,2}; H_1(\Sigma))$ defined in \cite{kawazumi1997gmm}. On the fiber $\Sigma=\pi^{-1}(0)\subset X$, pick a path $l$ such that  $l(0)\in S_2$ and $l(1)\in S_1$. Define $\omega_l$ by
\[\omega_l(\varphi):=\varphi(l)-l\in H_1(\Sigma).\]
Then we have
\begin{lemma}
\[
j_\varphi(e_3-e_4)=[\omega_l(\varphi)]\in \Coker(\varphi_*-1).
\]
\end{lemma}
\begin{proof}
Define a 2-chain $L: I\times I\to X$ by $L(s, t)=[l(s),t]$. Its boundary is given  by $-i_*(e_3)+\varphi(l)+i_*(e_4)-l\in B_1(X)$. Hence,
\[
i_*(e_3-e_4)=\iota_*([\varphi(l)-l])\in H_1(X)
\]
Since $\iota_*$ is injective, the lemma follows.
\end{proof}
From the lemma, we see the homolopy class $[\omega_l(\varphi)]\in \Coker(\varphi_*-1)$ is independent of the choice of the path $l$. If $\omega_l(\varphi)=0$, then $j_\varphi(e_3-e_4)=0$.  
\begin{remark}
If there exists a path $l$ from a point in $S_2$ to a point in $S_1$ which has no common point with the support of a representative of $\varphi\in\mathcal{M}_{g,2}$, then $m(\varphi)=[1:0]$. In particular, $m(id)=[1:0]$, the zero element of the monoid $\mathbf{QP}^1$.
\end{remark}
At the beginning of this section, we defined the commutative monoid structure on $\mathbf{QP^1}$. So integral multiples of $m(\varphi)$ are well-defined. 
\begin{proposition}
If $\varphi\in\mathcal{M}_{g,2}$ and $k\in \mathbf{Z}$, then
\[m(\varphi^k)=km(\varphi).\]
\end{proposition}
\begin{proof}
The proposition is trivial for $k=0$ and $k=1$. Assume $k\ge2$.

Let $m(\varphi)=[p:q]$. By the definition of $j_\varphi$, $pj_\varphi(e_3-e_4)=-q[S_1]\in \Coker(\varphi_*-1)$. Hence, there exists $v\in H_1(\Sigma)$ such that 
\[p[\varphi(l)-l]=-q[S_1]+(\varphi_*-1)v\in H_1(\Sigma)\]
Apply $\varphi^i$ ($i=1,2,\cdots k-1$) to the both sides of the equation and sum over $i$. Then 
\[
\sum_{i=1}^{k-1} p(\varphi^{i+1}(l)-\varphi^i(l))=-\sum_{i=1}^{k-1}\{[S_1]+(\varphi_*^{i+1}(v)-\varphi_*^i(v))\},
\]
that is
\[p(\varphi^k(l)-l)=-kq[S_1]+(\varphi_*^k-1)v.\]
Hence, $m(\varphi^k)=[p:kq]=km(\varphi)$ for $k>0$. 

By applying $\varphi^{-1}$ to the equation $p[\varphi(l)-l]=-q[S_1]+(\varphi_*-1)v$, we have
\[p[\varphi^{-1}(l)-l]=q[S_1]+(\varphi_*^{-1}-1)v\in H_1(\Sigma).\]
Hence, $m(\varphi^{-1})=[p:-q]=-m(\varphi)$. Since $m(\varphi^{-k})=-m(\varphi^k)=-km(\varphi)$ for $k>0$, the proposition follows for the case $k<0$.
\end{proof}
Now we compute the image of the function $m$. Especially we prove that $m$ is nontrivial. 
\begin{proposition}
For $g\geq 1$, $m$ is surjective. For $g=0$, $\Im(m)=[1:\mathbf{Z}]$.
\end{proposition}
\begin{figure}[h]
  \begin{center}
    \includegraphics{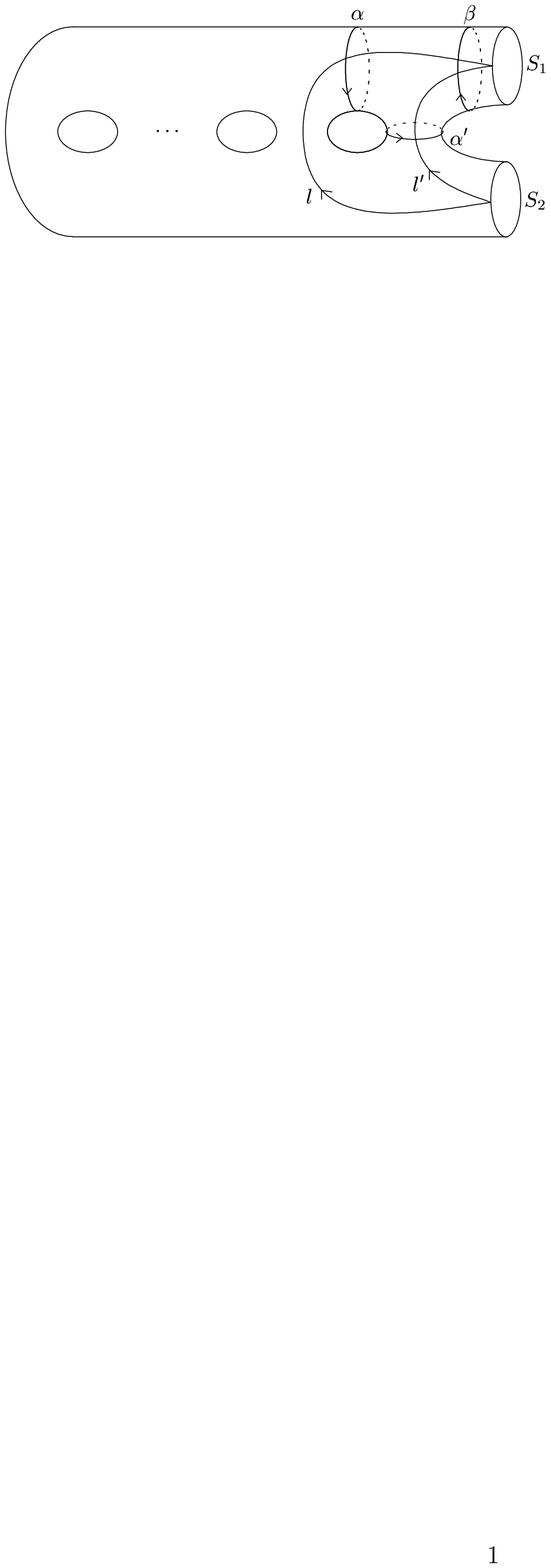}
  \end{center}
  \caption{}
  \label{d.eps}
\end{figure}
\begin{proof}
Suppose $g\geq1$. We choose oriented simple closed curves $\alpha$, $\alpha'$, and $\beta$ and paths $l$ and $l'$ as shown in Figure \ref{d.eps}. 
We denote the Dehn twists along a simple closed curve $C\subset\Sigma$ by $t_C$, and the homology class of $C$ by $[C]$ . Then $[\alpha]+[\alpha']+[\beta]=0\in H_1(\Sigma)$ since they bound a 2-chain. For $p\in\mathbf{Z}$, if we denote $\varphi:=t_{\alpha}^pt_{\alpha'}t_{\beta}^{-1}$, then
\begin{align*}
j_\varphi((p+1)(e_3-e_4))
&=
\omega_l(\varphi)+p\omega_{l'}(\varphi)\\
&=(t_{\alpha}^p t_{\alpha'}t_{\beta}^{-1})(l)-l+p\{(t_{\alpha}^p t_{\alpha'}t_{\beta}^{-1})(l')-l'\}\\
&=p([\alpha]+[\alpha']+[\beta])+[\beta]=[\beta]=[S_1].
\end{align*}
Hence, $j_{\varphi}((p+1)(e_3-e_4)-e_1)=0$, so that 
\[
m(\varphi)=[p+1:-1].
\]
By Proposition 2.5, we have
\[m(\varphi^{-q})=-q[p+1:-1]=
\left\{\begin{array}{ll}
[p+1:q], &\textrm{ if } p\ne-1\\
{}[0:1], &\textrm{ if } p=-1 .
\end{array} \right.
\quad (q\in \mathbf{Z})\]
Since $p$ and $q$ can run over all integers, we see $m$ is surjective for $g\ge 1$. 

For $g=0$, $\mathcal{M}_{0,2}$ is the infinite cyclic group generated by $t_\beta$.
Since $m(t_{\beta}^{-q})=[1:q]$, we have $\Im(m)=[1:\mathbf{Z}]$.
\end{proof}
\newpage
\section{The difference of two Meyer cocycles $\eta^*\tau_{g+1}$ and $\theta^*\tau_g$}
In this section (co)homology groups are with $\mathbf{Z}$ coefficient unless specified. 

Let $g\ge0$ be a positive integer. In Introduction, we defined the homomorphisms $\eta: \mathcal{M}_{g,2} \to \mathcal{M}_{g+1,0}$ and $\theta:\mathcal{M}_{g,2} \to \mathcal{M}_g$ to be the induced maps by sewing a pair of disks and by sewing an annulus onto the surface $\Sigma_{g,2}$ along their boundaries respectively. We denote the Meyer cocycle on the mapping class group of genus $g$ closed orientable surface $\mathcal{M}_g$ by $\tau_g\in Z^2(\mathcal{M}_g)$ and
define $\tilde{\tau}_g\in Z^2(\mathcal{M}_{g,2})$ to be the difference between the Meyer cocycles
\[
\tilde{\tau}_g:=\eta^*\tau_{g+1}-\theta^*\tau_g.
\]
Let $P:=S^2-\amalg_{i=1}^3 D^2$. In this section, we prove the main theorem and calculate the changes of signature associated with sewing a pair of trivial disk bundles $P\times \amalg_{i=1}^2 D^2$ and sewing an trivial annulus bundles $P\times (S^1\times I)$ onto $\Sigma_{g,2}$ bundle on the pair of pants $P$ along their boundaries. To state the main theorem, we define the sign of $[p:q]\in \mathbf{QP^1}$ by 
\[
\sign([p:q]):=
\begin{cases}
1 &\text{ if } pq>0,\\
0 &\text{ if } pq=0,\\
-1 &\text{ if } pq<0.
\end{cases}
\]
\begin{theorem}\label{theorem3.1}
For $\varphi, \psi\in\mathcal{M}_{g,2}$, we define 
\[\tilde{\phi}_g(\varphi):=\sign(m(\varphi)).\]
Then $\tilde{\phi}_g$ cobounds the difference $\tilde{\tau}_g$ between the Meyer cocycles $\eta^*\tau_{g+1}$ and $\theta^*\tau_g$ 
\begin{align*}
\tilde{\tau}_g(\varphi, \psi)&=\delta\tilde{\phi}_g(\varphi, \psi)\\
&=\sign(m(\varphi))+\sign(m(\psi))+\sign(m((\varphi\psi)^{-1})).
\end{align*}
\end{theorem}

\begin{remark}
Let $k$ be an integer.
By Lemma 2.2 and Proposition 2.5, $\tilde{\phi}_g$ has the properties
\begin{gather*}
\tilde{\phi}_g(\psi\varphi\psi^{-1})=\tilde{\phi}_g(\varphi), and\\
\tilde{\phi}_g(\varphi^k)=\sign(k)\tilde{\phi}_g(\varphi)
\end{gather*}
for any $g\ge0$.
\end{remark}
\subsection{Proof of Main Theorem}
In this subsection we prove Theorem \ref{theorem3.1}. 

In Introduction, we defined $E_{g,r}^{\varphi,\psi}$ as a $\Sigma_{g,r}$ bundle on the pair of pants $P$ which has monodromies $\varphi$, $\psi$, and $(\psi\varphi)^{-1}\in \mathcal{M}_{g,r}$ along $\alpha$, $\beta$, and $\gamma\in \pi_1(P)$ respectively, and in Subsection 2.1, we defined $X_{g,r}^\varphi$ by the mapping torus of $\Sigma_{g,r}\times I/\sim $ where $(x,1)\sim (h(x), 0)$ for $\varphi=[h]\in \mathcal{M}_{g,r}$.

We consider 
\[
E_{g+1}^{\eta(\varphi), \eta(\psi)}=E_{g,2}^{\varphi, \psi} \cup (-S^1\times I\times P),
\]
and
\[
X_{g+1}^{\eta(\varphi)}=X_{g,2}^\varphi \cup (-S^1\times I\times S^1).
\]
Define 
\[
\begin{array}{cccc}
G:&\partial D^2\times I&\to&\{1\}\times S^1\times I.\\
&(x, t)&\mapsto &(1, x, \frac{1+t}{3})
\end{array}
\]
By the map $G$, we can glue $D^2\times I$ to $I\times S^1\times I$ as shown in figure 3. 
\begin{figure}[h]
  \begin{center}
    \includegraphics{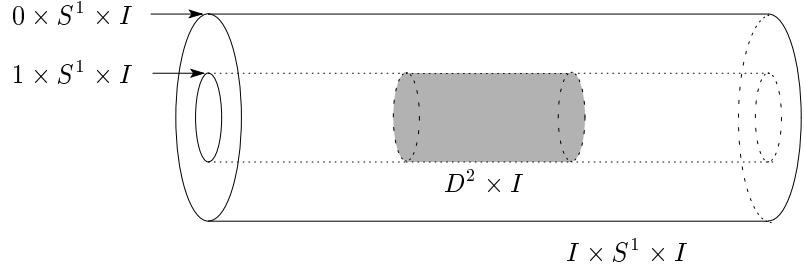}
  \end{center}
  \caption{Gluing map $G$}
  \label{g.eps}
\end{figure}
Glue $D^2\times I\times P$ to 
$I\times E_{g+1}^{\eta(\varphi), \eta(\psi)}=(I\times E_{g,2}^{\varphi, \psi}) \cup (-I\times S^1\times I\times P)$ with the gluing map 
$G\times id_P:\partial D^2\times I\times P\to \{1\}\times S^1\times I\times P$. In the same way, glue $D^2\times I\times S^1$ to 
$I\times X_{g+1}^{\eta(\varphi)}=(I\times X_{g,2}^\varphi) \cup (-I\times S^1\times I\times S^1)$ with the gluing map 
$G\times id_{S^1}:\partial D^2\times I\times S^1\to \{1\}\times S^1 \times I\times S^1$. 
Denote 
\[
\tilde{E}^{\varphi, \psi}:=(I\times E_{g+1}^{\eta(\varphi),\eta(\psi)})\cup (D^2\times I\times P), \text{ and }
\tilde{X}^\varphi:=(I\times X_{g+1}^{\eta(\varphi)})\cup (D^2\times I\times S^1).
\]
To prove main theorem, it suffices to prove Lemma \ref{lemma3.3} and Lemma \ref{lemma3.4} below.
\begin{lemma}\label{lemma3.3}
\[
(\eta^*\tau_{g+1}-\theta^*\tau_{g})(\varphi, \psi)=\Sign \tilde{X}^\varphi+\Sign \tilde{X}^\psi+\Sign \tilde{X}^{(\varphi\psi)^{-1}} \ \text{ for } \varphi, \psi\in\mathcal{M}_{g,2},\  g\geq0.
\]
\end{lemma}
\begin{lemma}\label{lemma3.4}
\[\Sign \tilde{X}^\varphi=\sign(m(\varphi)) \ \text{ for } \varphi\in\mathcal{M}_{g,2},\  g\geq0.\]
\end{lemma}
\begin{proof}[proof of Lemma 3.3]
Note that 
\[X^\varphi=\tilde{E}^{\varphi, \psi}|_{\partial D_1}. \]
Then we can see 
\begin{align*}
\partial \tilde{E}^{\varphi, \psi}
&=(\tilde{E}^{\varphi, \psi}|_{\partial D_1}\cup\tilde{E}^{\varphi, \psi}|_{\partial D_2}\cup\tilde{E}^{\varphi, \psi}|_{\partial D_3})
\cup E_{g}^{\theta(\varphi), \theta(\psi)}\cup -E_{g+1}^{\eta(\varphi), \eta(\psi)}\\
&=(\tilde{X}^{\varphi}\cup \tilde{X}^{\psi}\cup \tilde{X}^{(\psi\varphi)^{-1}})\cup E_{g}^{\theta(\varphi), \theta(\psi)}\cup -E_{g+1}^{\eta(\varphi), \eta(\psi)}.
\end{align*}
By Novikov Additivity, the fact $\Sign\partial\tilde{E}^{\varphi, \psi}=0$ implies
\[
\Sign(E_{g+1}^{\eta(\varphi), \eta(\psi)})-\Sign(E_{g}^{\theta(\varphi), \theta(\psi)})=\Sign \tilde{X}^\varphi+\Sign\tilde{X}^\psi+\Sign \tilde{X}^{(\psi\varphi)^{-1}}.
\]
Notice that $\tilde{X}^{(\psi\varphi)^{-1}}$ is diffeomorphic to $\tilde{X}^{(\varphi\psi)^{-1}}$, so that $\Sign \tilde{X}^{(\psi\varphi)^{-1}}=\Sign \tilde{X}^{(\varphi\psi)^{-1}}$.
By the definition of the Meyer cocycle, we have \[\Sign(E_{g+1}^{\eta(\varphi),\eta(\psi)})=\eta^*\tau_{g+1}(\varphi, \psi)
\text{, and }\ \Sign(E_{g}^{\theta(\varphi),\theta(\psi)})=\theta^*\tau_g(\varphi, \psi).
\]  
Define $\tilde{\phi}(\varphi)=\Sign(\tilde{X}^\varphi)$, then we have $\delta\tilde{\phi}=\eta^*\tau_{g+1}-\theta^*\tau_{g}$. We get the cobounding function $\tilde{\phi}$. 
\end{proof}
\begin{proof}[proof of Lemma 3.4]
Write simply $X:=X_{g+1}^{\eta(\varphi)}$, $X':=X_{g,2}^{\varphi}$, and $Y:=\tilde{X}^\varphi=(I\times X)\cup (D^2\times I\times S^1)$.

For $i=0,1$, define
\[
\begin{array}[t]{cccll}
j_i:&X&\to& I\times X &\hookrightarrow \ Y,\\
&x&\mapsto &(i, x)&\\
\end{array}
\]
where $I\times X \hookrightarrow Y$ is a natural embedding. We will prove there is a exact sequence 
\[
\begin{CD}
H_2(X')@>j_{0*}=j_{1*}>> H_2(Y)@>>> \Ker(H_1(\partial X')\to H_1(X'))@>>> 0.
\end{CD}
\]
Define $Y_1:=I\times X'$ and $Y_2:=(I\times S^1\times I\times S^1)\cup (D^2\times I\times S^1)\subset Y$, then 
\[
Y_1\simeq X', Y_2\simeq S^1, Y_1\cap Y_2\simeq \partial X'=(S_1\amalg S_2)\times S^1.
\]

By the Mayer-Vietoris exact sequence, we have 
\[
\begin{array}{cccccccc}
H_2(Y_1)\oplus H_2(Y_2)&\to& H_2(Y)&\to &H_1(Y_1\cap Y_2)&\to &H_1(Y_1)\oplus H_1(Y_2)& \textrm{(exact)}.\\
\rotatebox{90}{$\cong$}&&&& \rotatebox{90}{$\cong$}&& \rotatebox{90}{$\cong$}&\\
H_2(X')\oplus 0&&&& H_1(\partial X') && H_1(X')\oplus H_1(S^1)&
\end{array}
\]
Denote the map $H_1(\partial X')\to H_1(X')\oplus H_1(S^1)$ in the above diagram by $h$. the projection $H_1(\partial X')\to H_1(S^1)$ to the second entry of $h$ is the composite of inclusion homomorphism $H_1(\partial X')\to H_1(X')$ and $\pi_* :H_1(X')\to H_1(S^1)$. Therefore, 
\[
\Ker(H_1(\partial X')\to H_1(X')\oplus H_1(S^1))=\Ker(H_1(\partial X')\to H_1(X')).
\]
So the sequence is exact.

Next we construct the splitting $H_2(Y;\mathbf{Q})=j_{i*}H_2(X';\mathbf{Q})\oplus \Ker(H_1(\partial X';\mathbf{Q})\to H_1(X';\mathbf{Q}))$. Note that there exist $p$, $q\in \mathbf{Q}$ such that
\[
\Ker(H_1(\partial X';\mathbf{Q})\to H_1(X';\mathbf{Q}))=\mathbf{Q}(e_1+e_2)\oplus\mathbf{Q}\{p(e_3-e_4)+qe_1\}
\]
as in section 1.
To construct the splitting, we choose elements of inverse images of $e_1+e_2$, $p(e_3-e_4)+qe_1$ under $H_2(Y)\to H_1(\partial X')$. Define $\iota_Y: \Sigma_{g+1}\to Y$ by
\[
\begin{array}[t]{cccccccc}
&\Sigma_{g+1}&\to& X &\to& I\times X & \hookrightarrow &Y,\\
&x&\mapsto & (x,0)&\mapsto &(0,x,0)&&
\end{array}
\]
then we have
\[
\begin{array}{ccccc}
H_2(\tilde{X})&\to&H_1(Y_1\cap Y_2)&\to&H_1(\partial X'),\\
\iota_{Y*}[\Sigma_g]&\mapsto&\partial_*\iota_{Y*}[\Sigma_g]&\to &e_1+e_2
\end{array}
\]
so we choose $\iota_{Y*}[\Sigma_g]$ as an element of the inverse image of $e_1+e_2$.

Next, we choose an element of the inverse image of $p(e_3-e_4)+qe_1$. Since $p(e_3-e_4)+qe_1\in \Ker(H_1(\partial X';\mathbf{Q})\to H_1(X';\mathbf{Q}))$, there exists a singular 2-cochain $s\in C_2(X';\mathbf{Q})$ such that 
\[
\partial s=p(f_3-f_4)+qf_1\in B_1(X';\mathbf{Q}).
\]
For $i=0,1$, define $s'_{0i}:I\times S^1 \to I\times S^1\times I\times S^1\hookrightarrow Y_2$ by $s'_{0i}(s, t)=(i, 0, s, t)$.
then 
\[
[\partial s'_{0i}]=[j_if_3-j_if_4]\in H_1(Y_1\cap Y_2;\mathbf{Q}).
\]
\begin{figure}[h]
  \begin{center}
    \includegraphics{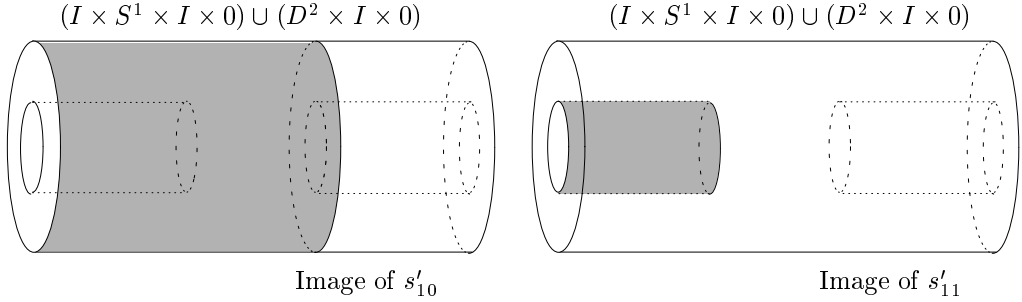}
  \end{center}
  \caption{Images of $s'_{10}$ and $s'_{11} \subset (I\times S^1\times I\times 0)\cup (D^2\times I\times 0)\subset Y$}
  \label{m.eps}
\end{figure}
Define $s'_{1i}: D^2\to (-I\times S^1\times I\times S^1)\cup(D^2\times I\times S^1) \subset Y$ as shown in Figure 4 by 
\begin{align*}
s'_{10}(x)=&\left\{
\begin{array}{lll}
(6x, 1, 0)&\in D^2\times I\times S^1&(||x||\leq\frac{1}{6}),\\
(2-6||x||, \frac{x}{||x||}, \frac{2}{3}, 0)&\in I\times S^1\times I\times S^1 &(\frac{1}{6}\leq||x||\leq \frac{1}{3}),\\
(0, 1-||x||, \frac{x}{||x||}, 0)&\in I\times S^1\times I\times S^1& (\frac{1}{3}\leq ||x|| \leq 1),
\end{array}\right.\\
s'_{11}(x, t)=&\left\{
\begin{array}{lll}
(\frac{3}{2}x, 0, 0)&\in D^2\times I\times S^1 &(||x||\leq\frac{2}{3}),\\
(1, \frac{x}{||x||}, 1-||x||, 0)&\in I\times S^1\times I\times S^1 &(\frac{2}{3}\leq ||x|| \leq 1).
\end{array}\right.
\end{align*}
Then, we have $[\partial s'_{1i}]=[j_if_1]\in H_1(Y_1\cap Y_2;\mathbf{Q})$.

Define $s'_i=ps'_{0i}+qs'_{1i}$, then it follows that
\[
[\partial s'_i]=[j_i(p(f_3-f_4)+qf_1)]\in H_1(Y_1\cap Y_2;\mathbf{Q}),
\]
so that we have $[\partial(j_is-s'_i)]=0 \in H_1(Y_1\cap Y_2;\mathbf{Q})$.

We see
\[
\begin{array}{ccccc}
H_2(Y;\mathbf{Q})&\to&H_1(Y_1\cap Y_2;\mathbf{Q})&\to&H_1(\partial X';\mathbf{Q}), \\
{}[j_i s-s'_i]&\mapsto& \partial_*[j_i s-s'_i]&\mapsto &p(e_3-e_4)+qe_1
\end{array}
\]
so that we can choose $[j_i s-s'_i]$ as an element of the inverse image of $p(e_3-e_4)+qe_1$.

Now we calculate the intersection form of $H_2(Y;\mathbf{Q})$. Define $X''_1=j_1(X)\cup (D^2\times 0\times S^1) \subset (I\times X)\cup (D^2\times I\times S^1)\subset Y$, then $X''_1$ is deformation retract of $Y$. Hence, every element of $H_2(Y;\mathbf{Q})$ is represented by a cocycle in $X''_1$. Therefore, a cohomology class is included in the annihilator of intersection form in $H_2(Y;\mathbf{Q})$ if it is represented by a cocycle which have no common point with $X''_1$. We see 
\[j_0(X')\cap X_1''=\emptyset \text{, and }\ \iota_Y(\Sigma_{g+1})\cap X''_1=\emptyset,\]
so that $\mathbf{Q}(e_1+e_2)$ and $j_{0*}H_2(X';\mathbf{Q})$ are included in the annihilator of intersection form in $H_2(Y;\mathbf{Q})$.

To describe the signature of $Y$, it suffices to calculate the self-intersection number of $[j_i s-s'_i]=p(e_3-e_4)+qe_1$.
The cocycle $j_is-s'_i$ satisfies
\begin{gather*}
\Im(j_0s)\cap(\Im(j_1s)\cup\Im(s'_{01})\cup\Im(s'_{11}))=\emptyset\\
\Im(s'_{00})\cap(\Im(j_1s)\cup\Im(s'_{01}))=\emptyset\\
\Im(s'_{10})\cap(\Im(j_1s)\cup\Im(s'_{01})\cup\Im(s'_{11}))=\emptyset,
\end{gather*}
so that 
\begin{align*}
(j_0s-s'_0)\cdot(j_1s-s'_1)
&=(j_0s-(ps'_{00}+qs'_{10}))\cdot(j_1s-(ps'_{01}+qs'_{11}))\\
&=ps'_{00}\cdot qs'_{11}.
\end{align*}
We can see $s'_{00}$ and $s'_{11}$ intersect only once positively. Hence, $\Sign(Y)=\Sign(pq)=\Sign(m(\varphi))$.
\end{proof}

\subsection{Wall's Non-additivity Formula}
Wall derives the Novikov additivity for a more general case: two compact oriented  smooth $4k$-manfolds are glued along a common submanifolds, which itself have boundary, of the boundaries of the original manifolds.

We will give the specific case of his formula for $k=1$:

Let $Z$ be a closed oriented smooth 2-manifold, $X_-$, $X_0$, $X_+$ compact oriented smooth 3-manifolds with the boundaries $\partial X_-=\partial X_0=\partial X_+=Z$, and $Y_-$, $Y_+$ compact oriented smooth 4-manifolds with the boundaries $\partial Y_-=X_-\cup_{Z}(-X_0)$, $\partial Y_+=X_0\cup_{Z}(-X_+)$. Here we denote by $M\cup_{B} (-N)$ the union of two manifolds $M$ and $N$ glued by orientation reversing diffeomorphism of their common boundaries $\partial M=\partial N=B$. Let $Y=Y_-\cup_{X_0} Y_+$ be the union of $Y_-$ and $Y_+$ glued along submanifolds $X_0$ of their boundaries. Suppose $Y$ is oriented by the induced orientation of $Y_-$ and $Y_+$.

Write $V=H_1(Z;\mathbf{R})$; let $A$, $B$, and $C$ be the kernels of the maps on first homology induce by the inclusions of $Z$ in $X_-$, $X_0$ and $X_+$ respectively.

We define 
\[
W:=\frac{B\cap(C+A)}{(B\cap C)+(B\cap A)},
\]
and a bilinear form $\Psi$ by
\[
\begin{array}{cccccc}
\Psi:&W&\times&W&\to&\mathbf{R}.\\
&(b&,&b')&\mapsto&b\cdot c'
\end{array}
\]
Here $c'$ is a element which satisfies $a'+b'+c'=0$, and $b\cdot c'$ denote the intersection product of $b$ and $c'$.

Then $\Psi$ is independent of $c'$ and well-defined on $W$. Denote the signature of the bilinear form $\Psi$ by $\Sign(V;BCA)$ and the signature of the compact oriented 4-manifold $M$ by $\Sign M$. We are now ready to state the formula. 
\begin{theorem}[Wall\cite{wall1969nas}]{\label{theorem:theorem1}}
$\Sign Y=\Sign Y_-+\Sign Y_+-\Sign(V;BCA)$.
\end{theorem}
\subsection{The differences of signature $\Sign E_{g}-\Sign E_{g,2}$ and $\Sign E_{g+1}-\Sign E_{g,2}$}
In this subsection, we calculate the difference of signature associated with sewing the trivial Disk bundles onto the $\Sigma_{g,2}$ bundle.

In Introduction, we defined $E_{g,r}^{\varphi,\psi}$ as a oriented $\Sigma_{g,r}$ bundle on $P$ which has monodromies $\varphi, \psi, (\psi\varphi)^{-1}\in\mathcal{M}_{g,r}$ along $\alpha, \beta, \gamma\in \pi_1(P)$.
If we fix $\varphi, \psi\in\mathcal{M}_{g,2}$, we denote simply   
\[
E_{g,2}:=E_{g,2}^{\varphi,\psi}, \hspace{1em}E_g:=E_{g}^{\theta(\varphi),\theta(\psi)}, \text{ and }\hspace{1em}E_{g+1}:=E_{g+1}^{\eta(\varphi),\eta(\psi)} \hspace{1em}(g\ge0).
\]
\begin{proposition}\label{proposition3.3}
$\Sign(E_g)-\Sign(E_{g,2})=-\Sign(m(\varphi)+m(\psi)+m((\varphi\psi)^{-1}))$
\quad $(g\ge0)$
\end{proposition}
\begin{proof}
$E_g$ is the union of $E_{g,2}$ and $E_D:=(D^2\amalg D^2)\times P$ glued along their boundaries. Using Non-additivity formula Theorem \ref{theorem:theorem1}, 
we calculate $\Sign(E_{g})-\Sign(E_{g,2})$.

Define $Y_-$, $Y_+$, $X_-$, $X_0$, $X_+$, and  $Z$ by
\begin{gather*}
Y_-:=(\amalg_{j=1}^2D^2)\times P,\quad Y_+:=E_{g,2},\ \\
X_-:=(\amalg_{j=1}^2D^2)\times \partial P,\quad X_+:=E_{g,2}|_{\partial P},\quad X_0:=(\amalg_{j=1}^2\partial D^2)\times P,\\
\text{and}\ Z:=(\amalg_{j=1}^2\partial D^2)\times \partial P, \quad\text{respectively}.
\end{gather*}
Here, by the notation stated in subsection \ref{subsection2.1}, 
\[
X_+=E_{g,2}|_{\partial P}\cong X^\varphi\amalg X^\psi\amalg X^{(\psi\varphi)^{-1}},\quad
Z\cong \partial X^\varphi\amalg \partial X^\psi\amalg \partial X^{(\psi\varphi)^{-1}}.
\]
Define $V$, $A$, $B$, and $C$ as stated in subsection 3.1.

Since $X^\varphi=X^\psi=X^{(\psi\varphi)^{-1}}=S^1\times S^1$, we can choose the base of $H_1(\partial X^\varphi;\mathbf{R})$, $H_1(\partial X^\psi;\mathbf{R})$, and $H_1(\partial X^{(\psi\varphi)^{-1}};\mathbf{R})$ as in section \ref{subsection2.1}. Denote their base by $\{e_{11}, e_{12}, e_{13}, e_{14}\}$, $\{e_{21}, e_{22}, e_{23}, e_{24}\}$, $\{e_{31}, e_{32}, e_{33},  e_{34}\}$ respectively. 

Since $Z=\partial X^\varphi\amalg \partial X^\psi\amalg \partial X^{(\psi\varphi)^{-1}}$, we think of $e_{ij}$ as an element of $H_1(Z;\mathbf{R})$. 

Denote $m(\varphi)=[a_1:b_1]$, $m(\psi)=[a_2:b_2]$, and $m((\psi\varphi)^{-1})=[a_3:b_3]$ respectively, then 
\begin{align*}
V&=H_1(Z,\mathbf{R})=\bigoplus_{i=1}^3\bigoplus_{j=1}^4\mathbf{R}e_{ij},\\
A&= \mathbf{R}e_{11}\oplus\mathbf{R}e_{21}\oplus\mathbf{R}e_{31}
\oplus\mathbf{R}e_{12}\oplus\mathbf{R}e_{22}\oplus\mathbf{R}e_{32},\\
B&=\mathbf{R}(e_{11}-e_{21})\oplus\mathbf{R}(e_{11}-e_{31})
\oplus\mathbf{R}(e_{12}-e_{22})\oplus\mathbf{R}(e_{12}-e_{32})\\
&\hspace{1em}\oplus\mathbf{R}(e_{13}+e_{23}+e_{33})
\oplus\mathbf{R}(e_{14}+e_{24}+e_{34}),\\
C&=\bigoplus_{i=1}^3\left\{
\begin{array}{ll}
\mathbf{R}(e_{i1}+e_{i2})\oplus\mathbf{R}(e_{i3}-e_{i4}+m_ie_{i1})&\text{if } a_i\ne0\\
\mathbf{R}e_{i1}\oplus\mathbf{R}e_{i2}& \text{if } a_i=0.
\end{array}
\right. \hspace{1em} \textrm{Here we denote } m_i:=\frac{b_i}{a_i}.
\end{align*}
Hence, 
\begin{align*}
B\cap A&=\mathbf{R}(e_{11}-e_{21})\oplus\mathbf{R}(e_{12}-e_{22})\oplus\mathbf{R}(e_{11}-e_{31})\oplus\mathbf{R}(e_{12}-e_{32}),\\
B\cap C&=\left\{
\begin{array}{ll}
\mathbf{R}(e_{11}-e_{21}+e_{12}-e_{22})&\\
\oplus\,\mathbf{R}(e_{11}-e_{31}+e_{12}-e_{32})&\text{if } a_i\ne0\hspace{1em}\text{for }i=1,2,3\\
\oplus\,\mathbf{R}(e_{13}+e_{23}+e_{33}-e_{14}-e_{24}-e_{34}+m_1e_{11}+m_2e_{21}+m_3e_{31})&
\text{and }m_1+m_2+m_3=0,  \\[5pt]
\mathbf{R}(e_{11}-e_{21}+e_{12}-e_{22})& \text{if }a_i\ne0\hspace{1em}\text{for } i=1,2,3 \\
\oplus\,\mathbf{R}(e_{11}-e_{31}+e_{12}-e_{32})&\text{and } m_1+m_2+m_3\ne0,\\[5pt]
\mathbf{R}(e_{11}-e_{21}+e_{12}-e_{22})\\
\oplus\,\mathbf{R}(e_{11}-e_{31}+e_{12}-e_{32})&\text{if }a_1=0, a_2\ne0, a_3\ne0,\\[5pt]
\mathbf{R}(e_{11}-e_{21})\oplus\mathbf{R}(e_{12}-e_{22})\\
\oplus\,\mathbf{R}(e_{11}-e_{31}+e_{12}-e_{32})&\text{if }a_1=a_2=0, a_3\ne0,\\[5pt]
\mathbf{R}(e_{11}-e_{21})\oplus\mathbf{R}(e_{12}-e_{22})\\
\oplus\,\mathbf{R}(e_{11}-e_{31})\oplus\mathbf{R}(e_{12}-e_{32})&\text{if }a_i=0\hspace{1em}\text{for } i=1,2,3,
\end{array}
\right.\\
B\cap (C+A)&=\left\{
\begin{array}{ll}
\mathbf{R}(e_{11}-e_{21})\oplus\mathbf{R}(e_{12}-e_{22})&\\
\oplus\,\mathbf{R}(e_{11}-e_{31})\oplus\mathbf{R}(e_{12}-e_{32})&\\
\oplus\,\mathbf{R}(e_{13}+e_{23}+e_{33}-e_{14}-e_{24}-e_{34})&\hspace{4cm}\,\text{if }a_i\ne0\hspace{1em}\text{for }i=1,2,3, \\[5pt]
\mathbf{R}(e_{11}-e_{21})\oplus\mathbf{R}(e_{12}-e_{22})&\\
\oplus\,\mathbf{R}(e_{11}-e_{31})\oplus\mathbf{R}(e_{12}-e_{32}),
&\hspace{4cm}\,\textrm{otherwise}.
\end{array}
\right.
\end{align*}
By computing the signature of $\Psi$, we have 
\begin{align*}
\Sign(V;BCA)=\left\{
\begin{array}{lr}
\Sign(m_1+m_2+m_3) &\text{if }a_i\ne0 \hspace{1em} \text{for }i=1,2,3,\\[5pt]
0&\textrm{otherwise}.
\end{array}
\right.
\end{align*}
Hence, 
\begin{align*}
\Sign(V;BCA)&=\Sign(m(\varphi)+m(\psi)+m((\psi\varphi)^{-1}))\\
&=\Sign(m(\varphi)+m(\psi)+m((\varphi\psi)^{-1})).
\end{align*}
By Non-additivity formula, we have 
\[
\Sign(E_{g})=\Sign(E_D)+\Sign(E_{g,2})-\Sign(V;BCA).
\]
Since $E_D$ is a trivial bundle $(D^2\amalg D^2)\times P$, we have $\Sign(E_D)=0$.

This completes the proof of the proposition.
\end{proof}
By the theorem and Proposition \ref{proposition3.3}, we can calculate the difference of signature $\Sign(E_{g})-\Sign(E_{g,2})$. 
\begin{corollary}
For $g\ge0$,
\begin{multline*}
\Sign(E_{g+1})-\Sign(E_{g,2})=\Sign(m(a))+\Sign(m(b))+\Sign(m((ab)^{-1}))\\
-\Sign(m(a)+m(b)+m((ab)^{-1})).
\end{multline*}
\end{corollary}
\bibliographystyle{amsplain}
\bibliography{first.bib} 

\providecommand{\bysame}{\leavevmode\hbox to3em{\hrulefill}\thinspace}
\providecommand{\MR}{\relax\ifhmode\unskip\space\fi MR }
% \MRhref is called by the amsart/book/proc definition of \MR.
\providecommand{\MRhref}[2]{%
  \href{http://www.ams.org/mathscinet-getitem?mr=#1}{#2}
}
\providecommand{\href}[2]{#2}
\begin{thebibliography}{10}

\bibitem{atiyah1987ldeta}
M.~Atiyah, \emph{{The logarithm of the dedekind $\eta$-function}},
  Mathematische Annalen \textbf{278} (1987), no.~1, 335--380.

\bibitem{endo2000mss}
H.~Endo, \emph{{Meyer's signature cocycle and hyperelliptic fibrations}},
  Mathematische Annalen \textbf{316} (2000), no.~2, 237--257.

\bibitem{harer1983shg}
J.~L. Harer, \emph{{The second homology group of the mapping class group of an
  orientable surface}}, Inventiones Mathematicae \textbf{72} (1983), no.~2,
  221--239.

\bibitem{harer1985shm}
\bysame, \emph{{Stability of the Homology of the Mapping Class Groups of
  Orientable Surfaces}}, The Annals of Mathematics \textbf{121} (1985), no.~2,
  215--249.

\bibitem{iida2004ale}
S.~Iida, \emph{{Adiabatic limit of $\eta$-invariants and the Meyer function of
  genus two}}, Master's thesis, University of Tokyo, (2004).

\bibitem{kasagawa2000fmc}
R.~Kasagawa, \emph{{On a function on the mapping class group of a surface of
  genus 2}}, Topology and its Applications \textbf{102} (2000), no.~3,
  219--237.

\bibitem{kawazumi1997gmm}
N.~Kawazumi, \emph{{A generalization of the Morita-Mumford classes to extended
  mapping class groups for surfaces}}, Inventiones Mathematicae \textbf{131}
  (1997), no.~1, 137--149.

\bibitem{korkmaz2003shg}
M.~Korkmaz and A.I. Stipsicz, \emph{{The second homology groups of mapping
  class groups of orientable surfaces}}, Mathematical Proceedings of the
  Cambridge Philosophical Society \textbf{134} (2003), no.~03, 479--489.

\bibitem{meyer1973sf}
W.~Meyer, \emph{{Die Signatur von Fl{\"a}chenb{\"u}ndeln}}, Mathematische
  Annalen \textbf{201} (1973), no.~3, 239--264.

\bibitem{morifuji2003msf}
T.~Morifuji, \emph{{On Meyer's function of hyperelliptic mapping class
  groups}}, J. Math. Soc. Japan \textbf{55} (2003), no.~1, 117--129.

\bibitem{powell1978ttm}
J.~Powell, \emph{{Two Theorems on the Mapping Class Group of a Surface}},
  Proceedings of the American Mathematical Society \textbf{68} (1978), no.~3,
  347--350.

\bibitem{wall1969nas}
CTC Wall, \emph{{Non-additivity of the signature}}, Inventiones Mathematicae
  \textbf{7} (1969), no.~3, 269--274.

\end{thebibliography}
\end{document}